\documentclass[11pt,leqno,twoside]{article}
\usepackage{geometry}
\usepackage{amsmath}
\usepackage{amsfonts}
\usepackage{amsthm}
\usepackage{indentfirst}

\usepackage{graphicx}
\usepackage{amssymb}
\usepackage{bm}
\usepackage{url}

\numberwithin{equation}{section}

\theoremstyle{plain}
\newtheorem{theorem}{\indent\rm T\,h\,e\,o\,r\,e\,m\;}[section]

\theoremstyle{definition}
\newtheorem{definition}[theorem]{\indent\rm D\,e\,f\,i\,n\,i\,t\,i\,o\,n\;}

\theoremstyle{remark}
\newtheorem{remark}[theorem]{\indent\rm R\,e\,m\,a\,r\,k\;}

\theoremstyle{exam}
\newtheorem{exam}[theorem]{\indent\rm E\,x\,a\,m\,p\,l\;e\;}

\renewenvironment{proof}{\indent\rm P\,r\,o\,o\,f.\;}{\hfill $\square$ \\ \indent}

\makeatletter                                                  %
\renewcommand*{\@seccntformat}[1]{
  \csname the#1\endcsname\;-                                   %
}                                                              %
\renewcommand{\section}{\@startsection{section}{1}{0mm}        %
   {1.5\baselineskip}
   {1\baselineskip}
   {\indent\normalfont\normalsize\bfseries}
   }                                                           %
\renewcommand*{\@seccntformat}[1]{
  \normalfont\bfseries\csname the#1\endcsname\;-               %
}                                                              %
\renewcommand\subsection{\@startsection                        %
  {subsection}{2}{0mm}
  {1.5\baselineskip}
  {1\baselineskip}
  {\indent\normalfont\normalsize\itshape}}
\renewcommand*{\@seccntformat}[1]{
  \normalfont\bfseries\csname the#1\endcsname\;-               %
}                                                              %
\renewcommand\subsubsection{\@startsection                     %
  {subsubsection}{2}{0mm}
  {1.5\baselineskip}
  {1\baselineskip}
  {\indent\normalfont\normalsize\texttt}}
\makeatother                                                   %
\begin{document}
\thispagestyle{empty}


\centerline{\large{\textbf{On the concept of center for geometric objects and related problems}}}

\begin{center}
{\sc M. Magdalena Martínez-Rico}, \
{\sc L. Felipe Prieto-Martínez}  \ {\small and}  \
{\sc R. Sánchez-Cauce}
\end{center}
\vspace {1.1cm}

\renewcommand{\thefootnote}{\fnsymbol{footnote}}


\renewcommand{\thefootnote}{\arabic{footnote}}
\setcounter{footnote}{0}

\vspace{0.6cm}
\begin{center}
\begin{minipage}[t]{11cm}
\small{
\noindent \textbf{Abstract.}
 In this work, we review the concept of center of a geometric object as an equivariant map, unifying and generalizing different approaches followed by authors such as C. Kimberling or A. Edmonds. We provide examples to illustrate that this general approach encompasses many interesting spaces of geometric objects arising from different settings. Additionally, we discuss two results that characterize centers for some particular spaces of geometric objects, and we pose five open questions related to the generalization of these characterizations to other spaces. Finally, we conclude this article by briefly discussing other \emph{central objects} and their relation to this concept of center.
\medskip

\medskip

\noindent \textbf{Mathematics~Subject~Classification:}
 51M04, 51M15.

}
\end{minipage}
\end{center}

\bigskip

\section{Introduction.}

In Geometry, the word \emph{center} is sometimes used  without rigorously specifying its meaning. What is the \emph{center} of a finite set of points in the plane? What is the \emph{geographical center} of a country?

For any integer $n\geq 1$, let $S(n)$ denote the group of similarities of $\mathbb R^n$ (with respect to the composition) and let us denote its identity element by $I$.  Let us recall, at this moment, the definition of group action. We say that $S(n)$ \textbf{acts} on a given set $\mathcal O$ if there is an operation $S(n)\times \mathcal O\to \mathcal O$ (we use, for $T\in S(n)$ and $O\in\mathcal O$, the notation $T(O)$ for this operation) satisfying that (1) $\forall O\in \mathcal O$ , $I(O)=O$ and (2)  $\forall T,S\in S(n)$, $\forall O\in \mathcal O$, $(T\circ S)(O)=T(S(O))$.

Let us propose the following formal definition of \emph{center}:

\begin{definition}  Let $\mathcal O$ be a set such that $S(n)$ acts on $\mathcal O$. A \textbf{center} is a  function $\mathfrak X:\mathcal O\to\mathbb R^n$ which is \textbf{equivariant}, that is, it conmutes with the action of $S(n)$: 
$$\forall T\in S(n),\;\forall O\in \mathcal O,\;T(\mathfrak{X}(O))=\mathfrak X(T(O)).$$

\noindent In this context, we say that $\mathcal O$ is the \textbf{space of geometric objects} of the center (so a center is a function that assign a point to each geometric object).

\end{definition}

The following example may illustrate which is the relation between what we expected for a concept of  \emph{center of a given geometric object} and equivariant maps.

\begin{exam} \label{exam} Let $\mathcal P_3'$ be the set of triangles in the plane (it will play the role of the space of geometric objects). This set can be identified with the set of subsets of $\mathbb R^2$ of size 3 whose elements are all of them distinct and non-collinear. 

For every element $\{A,B,C\}$ in $\mathcal P_3'$, there are four points in the plane which are remarkable. They are the centroid, the incenter, the circumcenter and the orthocenter. They were already familiar to the ancient Greeks and can be obtained with very simple constructions.

Note that, if $\{A,B,C\}$ and $\{A',B',C'\}$ are two similar triangles, $T$ is the similarity transformation satisfying $A'=T(A)$, $B'=T(B)$, $C'=T(C)$ and $\mathfrak C(\{A,B,C\})$ denotes the centroid (equiv. for the incenter, the circumcenter and the orthocenter), then $\mathfrak C(\{A',B',C'\})=T(\mathfrak C(\{A,B,C\}))$.

\end{exam}

The previous example may convince the reader of the pertinence of describing centers as equivariant maps and of the necessity of this property to achieve a coherent definition. But why a general theory should be needed? 

\begin{itemize}
\item C. Kimberling's studied in several papers (see \cite{K1,K2}, for instance) the concept of \emph{center} for the space of triangles in the plane. In his \emph{Encyclopedia of Triangle Centers} \cite{K} (more will be said about the author and the \emph{Encyclopedia} later) there is a list of 59421 different triangle centers. 

\item A. Edmonds studied the concept of center for the space of simplices in $\mathbb R^n$ (see \cite{E}). 

\item In \cite{VT} the reader may find a lot of material of what the author calls \emph{quadri-figures}.

\item In \cite{PS} centers of polygons are studied. 

\end{itemize}

\noindent All of these authors  (and many others that, in all likelihood, also deserve to appear in this list) independently studied centers for a particular choice of the space of geometric objects. The main goal of this paper is to show that, with the above unifying definition of center, there are some properties and problems common to every possible choice of the space of geometric objects.

\begin{remark} Note that, in the definition of center, we have not imposed any axiom of continuity for the function $\mathfrak X$. This is done by some other authors, such as A. Edmonds in \cite{E}. We will talk about this later.

\end{remark}

This paper is structured as follows. In the following section, it is justified, by showing several examples, that the general approach presented above for the concept of center applies to many interesting spaces of geometric objects (arising from Pure and Applied Mathematics). This section is basic for the rest of the paper.  In Sections \ref{fq} and \ref{sq} the reader may find a discussion concerning two general questions of this theory of centers, that have already been explored for some particular cases of the space of geometric objects (typically, for the set of triangles) but are still open for some others.  

\begin{itemize}

\item The first one deals with the existence of {a} simple expression for the centers. This type of expressions are well known for triangles (we refer, again, to Kimberling's \emph{Encyclopedia}). For example, the centroid of a triangle $\{A,B,C\}$ can be obtained as
$$\mathfrak C(\{A,B,C\})=\frac{1}{3}A+\frac{1}{3}B+\frac{1}{3}C. $$

\item The second, deals with coincidence of centers. For a fixed geometric object  $O\in\mathcal O$, there is a relation between coincidence of centers and the group of symmetries  of $O$. For example, we have the following well known result, which is part of a principle that can be, in some sense, generalized for any other space of geometric objects:

\end{itemize}

\begin{theorem}[see \cite{F,I}] \label{theorem.es} For the space of triangles $\mathcal P_3'$ described before:

\begin{itemize} 

\item[(a)] The incenter and the orthocenter coincide if and only if the triangle is equilateral.

\item[(b)] The incenter, centroid and orthocenter are collinear if and only if the triangle is isosceles.

\end{itemize}


\end{theorem}

\noindent  Finally, in the last section, we briefly discuss the possibility of describing other \emph{central objects}, that is, equivariant maps whose domain is, again, a space of geometric objects, but whose codomain is other space (endowed with a $S(n)$-action).

\section{Some examples.}

In the following, we may find some examples of centers. Each subsection studies a different choice of the space of geometric objects: multisets, polygons, Borel sets and others.

\subsection{Centers of multisets.}

For this subsection, for $n,m\geq 1$, let $\mathcal M_m$ denote the set of m-submultisets of  $\mathbb R^n$ (multisets of elements in $\mathbb R^n$ of size $m$).

The most simple center that can be defined for this space of geometric objects is the following:

\begin{exam} \label{example.centroid} The  \textbf{centroid} $\mathfrak C:\mathcal M_m\to\mathbb R^n$ (we will always use $\mathfrak C$ to denote the centroid, even for different choices of the space of geometric objects) is defined as
$$\mathfrak C(\{X_1,X_2,\ldots, X_m\})=\frac{1}{m}X_1+\frac{1}{m}X_2+\ldots+\frac{1}{m}X_m.$$

\end{exam}

Note that, for some authors, the centroid of a $m$-multiset $O$ is called \emph{the center of $O$}, but there are many other interesting centers for this space of geometric objects. Some examples can be motivated, in the case $n=2$, from \emph{facility location problems}. These are classic optimization problem whose statement is similar to the following: \emph{determine the best location for a factory or warehouse to be placed, based on transportation distances, provided the customers locations.}

\begin{exam} For every $O\in\mathcal M_m$, let us define $B_O$ to be the smallest $n$-dimensional ball containing $O$. The problem of finding this $B_O$ (provided $O$) is  known as the \textbf{Smallest Enclosing Ball Problem} and it can be proved  straightforward that has a unique solution (we refer the reader to \cite{W} for more information). We define the \textbf{1-center} to be the equivariant map $\mathfrak F:\mathcal M_m\to\mathbb R^n$ such that, for each $O\in \mathcal M_m$, $\mathfrak F(O)$ is the center of $B_O$. In the context of facility location problems, in the case $n=2$, for a set of customers placed in the points in $O$, $\mathfrak F(O)$ is the optimal point to place a factory if we want to minimize the longest distance from the facility to the customers.

\end{exam}

\begin{exam} \label{example.medoid} For every $\{X_1,\ldots,X_m\}\in\mathcal M_m$, let us consider the following optimization problem:
$$\min_{P\in\mathbb R^n}\sum_{i=1}^m\|P-X_i\|. $$

\noindent It is known that for non-adversarially constructed data, the solution of the problem above is almost certainly unique (see \cite{BT}). Let us denote by $\mathcal M_m'$ to be the subset of $\mathcal M_m$ consisting in these elements for which the solution is unique. We define the \textbf{medoid} $\mathfrak M:\mathcal M_m'\to\mathbb R^n$ to be the map given by
$$\mathfrak M(\{X_1,\ldots, X_m\})=\min_{P\in\mathbb R^n}\sum_{i=1}^m\|P-X_i\|. $$

\noindent In the context of facility location problems in the case $n=2$, for a set of customers placed in the points in $O$, $\mathfrak M(O)$ is the optimal point to place a factory if we want to minimize the average distance from the facility to the customers.

\end{exam}

\subsection{Centers of polygons.} 

For this subsection, $n=2$ and $\mathcal P_m$ will denote the set of polygons with $m$ vertices, for $m\geq 3$. The elements in $\mathcal P_m$ cannot be identified with multisets, because the information of which vertex is adjacent to which other matters, in this case. A possibility is to use sequences $(X_1,\ldots, X_m)\in(\mathbb R^2)^m$, provided that

\begin{itemize}

\item[(1)] in this notation, for every $i=1,\ldots,i-1$, $X_i$ is adjacent to $X_{i+1}$ and $X_m$ is adjacent to $X_1$, and

\item[(2)]  two sequences $(X_1,\ldots, X_m)$, $(X_1',\ldots, X_m')$ \emph{represent the same polygon} (with a different labelling of the vertices) if there exists a permutation of $\alpha$ such that, for all $i=1,\ldots, m$,  $X_i'=X_{\alpha(i)}$ and such that $\alpha$ is in the group generated by the permutations $\rho,\sigma$ of $\{1,\ldots, m\}$, where, for all $i\in\{1,\ldots, m\}$,
$$\rho(i)=i+1\mod m,\qquad \sigma(i)=m-i.$$

\end{itemize}

\noindent In other words, $\mathcal P_m$ is a quotient set of $(\mathbb R^2)^m$ and its elements will be denoted as $[(X_1,\ldots, X_m)]$. In this notation, the \textbf{sides} of a polygon identified with the sequence $(X_1,\ldots, X_m)$ are those segments with endpoints $X_{i}$, $X_{i+1}$, for $i=1,\ldots, m$, considering subscripts modulo $m$.

The definition of polygon appearing above is the one appearing in \cite{BS} (in this book the term $m$-gon is used, instead) and in \cite{PS}. It is very general and includes pathological examples of polygons. There are some subsets of $\mathcal P_m$ which may be more suitable to be chosen as the space of polygons. For instante, let us define $\mathcal P_m'\subset \mathcal P_m$ to be the set of polygons whose sidelengths are strictly greater than 0 and that are \textbf{simple}, that is, the sides do not intersect.

\begin{exam} Note that the {centroid} can also be defined in this setting as
$$\mathfrak C(\{X_1,\ldots, X_m\})=\frac{1}{m} X_1,\ldots,\frac{1}{m}X_m.$$

\end{exam}

Again, this is not the only possible center of interest. Some of them can be found in \cite{PS} but let us, at least, introduce the following two articifial academic examples:

\begin{exam}  Let $\mathfrak S:\mathcal P_m'\to \mathbb R^2$ defined by 
$$\mathfrak S(X_1,\ldots, X_m)=\frac{d_{1,m}+d_{1,2}}{2p}X_1+\frac{d_{21}+d_{23}}{2p}X_2+\ldots+\frac{d_{m,m-1}+d_{m1}}{2p}X_m, $$

\noindent where $d_{i,j}$ denotes the distance between $X_i,X_j$ and $p$ denotes the perimeter of $(X_1,\ldots,X_m)$ (note that the definition does not depend on the labelling of the vertices).

\end{exam}

\begin{exam} Let $\mathfrak A:\mathcal P_m'\to \mathbb R^2$ defined by 
$$\mathfrak A(X_1,\ldots, X_m)=\frac{\alpha_1}{(m-2)\pi}X_1+\frac{\alpha_2}{(m-2)\pi}X_2+\ldots+\frac{\alpha_m}{(m-2)\pi}X_m, $$

\noindent where $\alpha_i$, for $i=1,\ldots, m$ denotes the value of the interior angle of $(X_1,\ldots, X_m)$ (note, again, that the definition does not depend on the labelling of the vertices).

\end{exam}

\subsection{Centers of Borel sets.} Let us recall that \textbf{Borel sets} are those sets in a topological space, $\mathbb R^n$ in the following, that can be obtained from open sets (or, equivalently, from closed sets) through the operations of countable union, countable intersection, and relative complement. Borel sets are $H^d$-measurable, where $H^d$ denotes the $d$-dimensional Hausdorff measure. Let $\mathcal B_d$ denote the set of Borel sets in $\mathbb R^n$ with positive and finite $d$-dimensional Hausdorff measure.

\begin{exam} \label{example.supercentroid}  In this setting, we can also define the centroid $\mathfrak C:\mathcal B_d\to\mathbb R^n$  by
$$\mathfrak C(O)=\left(\frac{\displaystyle{\int_Ox_1dH^d}}{\displaystyle{\int_OdH^d}},\ldots,\frac{\displaystyle{\int_Ox_ndH^d}}{\displaystyle{\int_OdH^d}}\right). $$

\end{exam}

Note that, for the case $d=0$, the Hausdorff measure $H^0$ of a given set $O$ is just its number of elements. So the formula for the centroid appearing in Example \ref{example.supercentroid} coincides with the one in Example \ref{example.centroid}.

In the case $d=n$, then Hausdorff measure is proportional to Lebesgue measure. In this particular case, many abuses of notation are commited in the literature. \emph{Geographical centers} are illustrative examples. The problem of determined such centers have received, historically, a lot of attention through the centuries. There are a lot of places claiming to be the center of a region, but the question is \emph{with respect to which definition?} At the beginning, the regions were considered to be plane, so the problem fits in the case $n=2$ of the setting above. Next, we show some examples of this:

\begin{itemize}

\item For United States, the \emph{geographic center} of area is defined as \emph{the point at which the surface of the United States would balance if it was a plane of uniform weight per unit of area} \cite{USA}. So, once the approximation with a plane region is made (which is not a trivial matter), this \emph{geographic center} coincides with the centroid as defined in Example \ref{example.supercentroid}. According to this definition, two historical geographic centers can be identified. The first is the geographic center of the conterminous United States, located near Lebanon, Kansas. This center was the official one from 1912 (with the admissions of New Mexico and Arizona to the 48 contiguous United States) to 1959 (with the admissions of Alaska and Hawaii). At this point, the geographic center moved to the north of Belle Fourche, South Dakota. This point is the one currently accepted as center of the United States by the United States Coast and Geodetic Survey and the U.S. National Geodetic Survey.

\item Officially, there is no center of Australia, because, as explained in \cite{Aust}, there is no unique accepted method to compute the center of a large, curved, irregularly-shaped area. Indeed, there are five possible center points of Australia: the \emph{center of gravity}, the \emph{Lambert gravitational center}, the \emph{furthest point from the coastline}, the \emph{geodetic median point}, and the \emph{Johnston Geodetic Station}. Although, the most plausible one seems to be the Lambert gravitational center, which was determined in 1988 by the Royal Geographical Society of Australasia.

\item Traditionally, the Cerro de los Ángeles, located in Getafe (Madrid), is considered the \emph{geographic center} of the Iberian Peninsula, although the criterion used to claim that is unclear. However, in the past, the town of Pinto (Madrid) held that title; in fact, it is believed that the name \emph{Pinto} comes from \emph{Punctum}, from the Latin \emph{point}, possibly named as such because it was considered the central territory of the peninsula by the romans. Recently, there has been a controversy about the peninsular geographic center, as some people argue that it is located in Méntrida (Toledo), not in Madrid.

\item Finally, we can consider the \emph{poles of inaccessibility} as other particular cases of centers. These points refer to locations that are the most challenging to reach according to a geographical criterion of inaccessibility \cite{G}. They are usually refered to the most distant points from the coastline, implying a \emph{maximum degree of continentality} or \emph{oceanity}, and they are defined as the centers of the largest circles that can be drawn within an area of interest without encountering a coast. One of the most famous poles of inaccessibility is the Oceanic pole of inaccessibility, also known as \textit{Point Nemo}, which is the place in the ocean that is farthest from land. Reciprocally, we find the continental pole of inaccessibility, which is the point farthest away from the coast, located in Eurasia, in the northwest of China.

\end{itemize}

For the case $n=2$, another interesting example is the one of the set $\mathcal B_1'$ of \textbf{rectificable Jordan curves} (simple and closed curves that have finite length, that is, are elements in $\mathcal B_1$) enclosing an element in $\mathcal B_2$. This case is, we could say, the limit case of $\mathcal P_m'$, for $m\to\infty$: we have not only sets, but sets with some extra structure. Note that, for some element $c\in\mathcal B_1'$ and for $B\in\mathcal B_2$ being its interior region, the centroid of $c$ {does not necessarily need to coincide} with the centroid of $B$ (see Figure 1).

\begin{figure}[h!]
\centering
\includegraphics[width=0.7\textwidth]{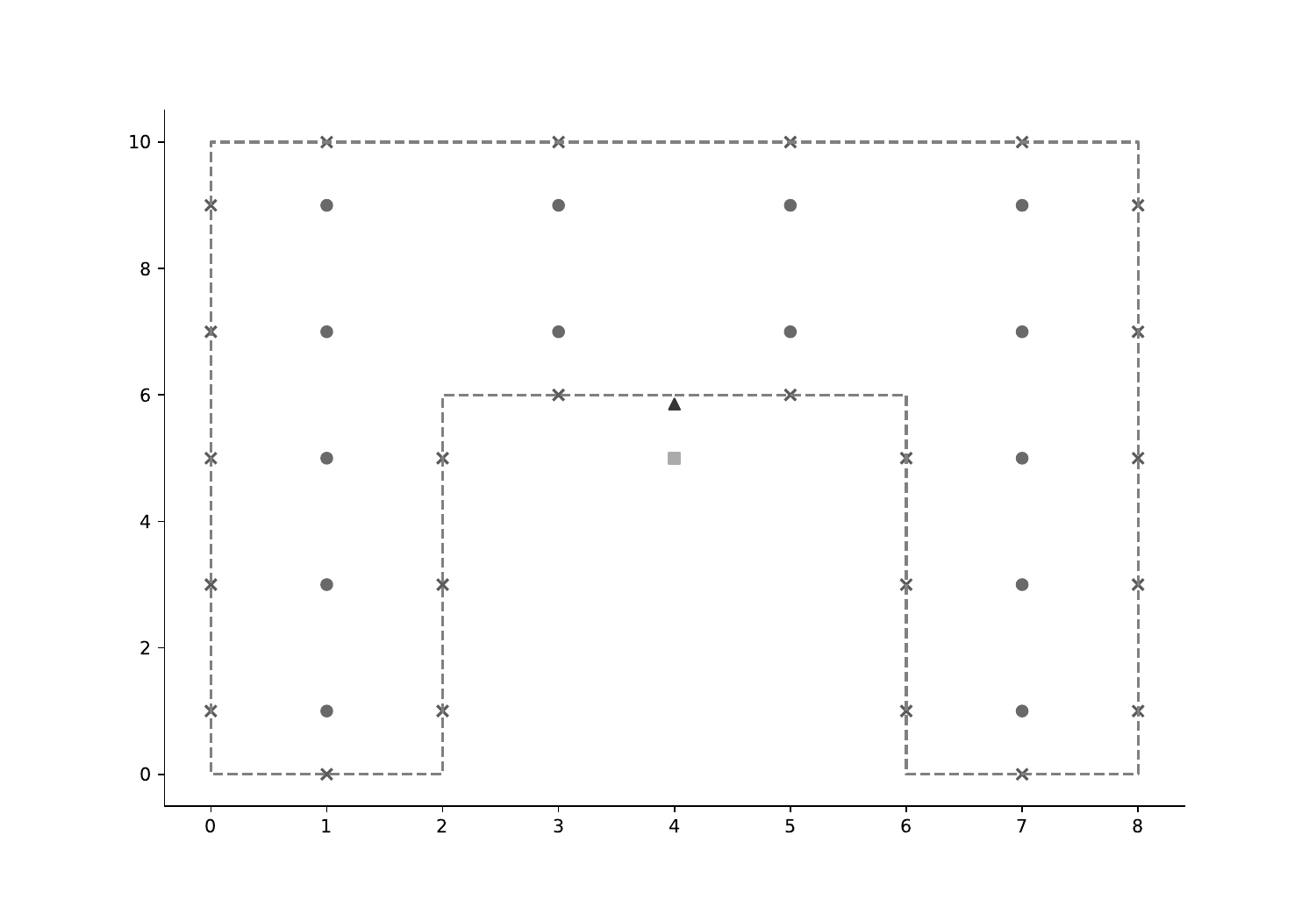}
\caption{A well known method for computing centroids is the one of \emph{geometric decomposition}. The centroid of the interior region (as a center of a Borel set) coincides with the centroid of the points represented with ($\bullet$) (as a center of a multiset). The centroid of the border (as a centroid of a Borel set) coincides with the centroid of the points represented with ``$\times$'' (as a centroid of the multiset). These centroids are marked with a triangle ($\blacktriangle$) and a square ($\blacksquare$), respectively, and do not coincide.}

\end{figure}

\subsection{Centers with other type of spaces of geometric objects.} {We finish this section showing a last example of center that appears, in particular, in the medical domain.}

\begin{exam} An interesting choice of the space of geometric objects $\mathcal O$, arising from Applied Mathematics, is the one consisting in those multisets of $m$ points in $\mathbb R^m$, such that each element has an associated weigth, that is, 
$$\mathcal O=\{\{(X_1,w_1),\ldots,(X_m,w_m):X_i\in\mathbb R^n, w_i\in\mathbb R_{>0}\}\}.$$

\noindent In this case, the \textbf{weighted centroid} $\mathfrak W:\mathcal A\to\mathbb R^n$ is defined as
$$\{(X_1,w_1),\ldots,(X_m,w_m)\}=\frac{w_1}{w_1+\ldots+w_m}X_1+\ldots+\frac{w_m}{w_1+\ldots+w_m}X_m. $$

\noindent This center appears, among many other applications, in problem related to the determination of the centroid of a tumor from its digital images (case $n=3$ or, in some experiments, $n=2$). The determination of this center is, in turn, useful to study cancer prognosis (based on the study of the glucose consumption of the cells \emph{near the center} and \emph{far from the center}). In this case, for some algorithms, the weigths are equal to a power of the inverse of time-of-flight within the breast associated to the pulse used for the image (see \cite{LGV} for more information).

\end{exam}

\section{First Question.} \label{fq}

Note that, in most of the examples appearing in the previous section, centers have a \emph{simple expression}. 


Let us start with the space of geometric objects $\mathcal P_3'$, described in Example \ref{exam}. C. Kimberling noticed that every center $\mathfrak X:\mathcal P'_3\to\mathbb R^2$ admits an expression 
$$\mathfrak X(\{A,B,C\})=g_1A+g_2B+g_3C,$$

\noindent where $g_1,g_2,g_3$ are certain functions depending on $A,B,C$ satisfying some properties of symmetry. The situation is summarized in the following result. To simplify the statement, let us explain at this moment that, for us, a function is \textbf{positively homogeneous} if when all its arguments are multiplied by a positive scalar, then its value is multiplied by some power of this scalar. Also, a function $g$ is  \textbf{invariant with respect to isometries} if, for every isometry $T$, we have that $g=g\circ T$ (provided that there is a natural action of the group of isometries on the domain of $g$).

\begin{theorem} \label{theorem.expression} A function $\mathfrak X:\mathcal P_3'\to\mathbb R^2$ is a center if and only if it satisfies any (and then all of them) of the following equivalent properties:

\begin{itemize}

\item[(a)] Let $\mathcal D=\{(a,b,c)\in\mathbb R^3: a< b+c,b<a+c,c<a+b)\}$, then there exists a  function  $g:\mathcal D\to\mathbb R$  such that, $\forall\{A,B,C\}\in\mathcal A'$, $\mathfrak X(\{A,B,C\})$ equals to
\begin{equation}\label{eq.t1}\textstyle{\frac{1}{g(a,b,c)+g(b,c,a)+g(c,a,b)}}(g(a,b,c)A+g(b,c,a)B+g(c,a,b)C), \end{equation} 

\noindent where $a$ denotes the length $BC$, $b$ denotes the length $AC$ and $c$ denotes the length $AB$, and satisfies the following properties: it is positively homogeneous and it satisfies the symmetry $g(a,b,c)=g(a,c,b)$, for all $(a,b,c)\in\mathcal D$.

\item[(b)] Let $\mathcal D=\{(A,B,C):\{A,B,C\}\in\mathcal P_3'\}$, then there exists a  function  $g:\mathcal D\to\mathbb R$  such that, $\forall\{A,B,C\}\in\mathcal P_3'$, $\mathfrak X(\{A,B,C\})$ equals to
\begin{equation}\label{eq.t2}\textstyle{\frac{1}{g(A,B,C)+g(B,C,A)+g(C,A,B)}}(g(A,B,C)A+g(B,C,A)B+g(C,A,B)C)\end{equation}

\noindent  and satisfies the following properties: it is positively homogeneous, it is invariant with respect to isometries and it satisfies the symmetry $g(A,B,C)=g(A,C,B)$, for all $(A,B,C)\in\mathcal D$.

\item[(c)] Let $\mathcal D=\{(X,O):O\in\mathcal P_3',X\in O\}$ (sets in $\mathcal P_3'$ with a distinguished element), then there exists a  function  $g:\mathcal D\to\mathbb R$  such that, $\forall\{A,B,C\}\in\mathcal P_3'$, $\mathfrak X(\{A,B,C\})$ equals to
\begin{equation} \label{eq.t3} \textstyle{\frac{1}{g(A,O)+g(B,O)+g(C,O)}}(g(A,O)A+g(B,O)B+g(C,O)C)\end{equation}

\noindent  and satisfies the following properties: it is postively homogeneous and it is invariant with respect to isometries.

\end{itemize}

\end{theorem}

\begin{proof} Some parts of this proof can be found in \cite{K1,FP,PS} but, for the sake of completeness, let us include a sketch of the main arguments herein.

If $\mathfrak X(\{A,B,C\})$ is defined, for every $\{A,B,C\}\in\mathcal P_3'$, as in any of the Equations \eqref{eq.t1}, \eqref{eq.t2} or \eqref{eq.t3}, then it is a center. The reason is that  these expressions are of the type 
$$\lambda_1 A+\lambda_2B+\lambda_3C,\qquad \text{with }\lambda_1+\lambda_2+\lambda_3=1, $$

\noindent so $(\lambda_1,\lambda_2,\lambda_3)$ are barcentric coordinates in the reference $(A,B,C)$. Finally, note that, in each case and for every similarity $T$,  the coordinates associated to the reference $(T(A),T(B),T(C))$ are equal to those associated to $(A,B,C)$.

On the other hand, suppose that we start with a center $\mathfrak X:\mathcal P_3'\to\mathbb R^2$. For each $\{A,B,C\}\in\mathcal P_3'$, let us consider the barycentric coordinates of $\mathcal X(\{A,B,C\})$ in the reference $(A,B,C)$:
$$\mathcal X(\{A,B,C\})=\lambda_1A+\lambda_2B+\lambda_3C. $$

\noindent In each of the cases described in (a), (b) and (c), we can check that the barycentric coordinates can be considered as functions defined in the corresponding domain $\mathcal D$ and that must satisfy the corresponding conditions and symmetries, provided that $\mathfrak X$ is a center.
\end{proof}

\begin{remark} We omit details but, in the previous theorem, the center is continuous if and only if the corresponding function $g$ (in each case) is continuous.

\end{remark}

In general,  analogues for  the spaces of geometric objects $\mathcal M_m,\mathcal P_m'$ are slightly more complicated, but still possible. The reason is that, in an expression of the type
$$\mathfrak X(O)=\lambda_1V_1+\ldots+\lambda_nV_n,\qquad\lambda_1+\ldots+\lambda_n=1, $$

\noindent the coefficients $(\lambda_1,\ldots,\lambda_n)$ are not \emph{coordinates} (that is, they are not univocally determined). Some results concerning the case $\mathcal P_m'$ can be found  in the manuscript \cite{FP}:

\begin{theorem}[Theorem 2.3 in \cite{FP}] For every element $(V_1,\ldots, V_n)$ in $(\mathbb R^2)^n$, let us consider the matrix $\bm M_1=[d_{ij}]_{1\leq i,j\leq n}$ such that $d_{ij}$ is the distance between $V_i,V_j$ (that is, the corresponding \textbf{distance matrix}). Let $\mathcal D$ denote the set of those matrices corresponding to elements in $\mathcal P_m'$ (formally, to representatives of equivalence classes in $\mathcal P_m'$ since, for the matrix, the choice of the vertex labelled as first one matters, but for the polygon it does not). A function $\mathfrak X:\mathcal P_m'\to\mathbb R^2$ is a center if and only if there exists a function $g:\mathcal D\to\mathbb R$ such that, for every $[(V_1,\ldots, V_n)]\in\mathcal P_m'$,
$$\mathfrak X([(V_1,\ldots, V_n)])=\lambda_1V_1+\ldots+\lambda_nV_n, $$

\noindent where, for $k=1,\ldots, n$, 
\begin{equation}\label{eq.lambda}\lambda_k=\frac{g(\bm M_k)}{g(\bm M_1)+\ldots+g(\bm M_n)},\end{equation}

\noindent {being} $\bm M_k$  the $n\times n$ matrix whose entry $(i,j)$ equals $d_{\rho^{k-1}(i)\rho^{k-1}(j)}$ and $g(\bm M_1)+\ldots+g(\bm M_n)\neq 0$.

\end{theorem}

The argument in the proof of this theorem, appearing in \cite{FP}, contains ideas that can be applied for the case $\mathcal M_m$ (which is the case $d=0$ of the Open Question 1 below).

\begin{remark} Note, also, that proving the existence of such an expression does not mean that the functions $g$ are easy to obtain. A well known example, that also illustrate the idea of these general problems appearing in many areas of Mathematics,  is the case in Example \ref{example.medoid} for the space of geometric objects $\mathcal M_m$. This example already appears in \cite{PS} and in \cite{B} the reader may find a discussion, {using} other notation, for the fact that it is not easy to find this $g$.

\end{remark}

Up to which point can we extend this principle stating that all the centers can be obtained from a \emph{simple expression}? In particular, we should clarify what is a \emph{simple expression} for each space of geometric objects. For example is also possible to state an analogue (that, up to our knowledge, has not been studied) of Statement (c) in Theorem \ref{theorem.expression} for $\mathcal B_d$:

\medskip

\noindent \textbf{Open Question 1.} Prove the following statement. A function $\mathfrak X:\mathcal B_n\to\mathbb R^n$ is a center if and only if, for $\mathcal D=\{(X,O):O\in\mathcal B_d,X\in O\}$, there exists a function $g:\mathcal D\to\mathbb R$, which is positively homogeneous, invariant with respect to isometries and such that, for all $O\in\mathcal B_d$, we have that:
$$\mathfrak X(O)=\left(\frac{\int_Ox_1g(X,O)dH^d}{\int_O g(X,O)dH^d},\ldots,\frac{\int_Ox_ng(X,O)dH^d}{\int_O g(X,O)dH^d} \right)  $$

\noindent (where $x_1,\ldots, x_n$ denote the corresponding coordinate of each point in $O$). Moreover, $\mathfrak X$ is continuous, if and only if $g$ is continuous (in this case, a discussion of the notion of \emph{continuity} for $\mathfrak X$ {and} $g$ should be included).

\section{Second Question.} \label{sq}

Now we will show a result related to Theorem \ref{theorem.es}  which is valid for any space of geometric objects. This argument, that uses de Axiom of Choice, already appears in  the manuscript \cite{P} but, for the sake of completeness, we will reproduce and generalize it here in our notation:

\begin{theorem} \label{theo.equifacetal} Let $\mathcal O$ be a space of geometric objects with the property that, if the set of objects in $\mathcal O$ fixed by an element $T\in S(n)$ is non-empty, so is the set of points in $\mathbb R^n$ fixed by $T$. Let $O\in \mathcal O$. For every point $P\in\mathbb R^n$, there is a center $\mathfrak X:\mathcal O\to \mathbb R^n$ such that $\mathfrak X(O)=P$ if and only if $P$ is fixed by the group of symmetries of $O$ (that is, for every $T\in G$ such that $T(O)=O$, we have that $T(P)=P$).

\end{theorem}

\begin{proof} It is clear that, for every center $\mathfrak X$, for every $O\in\mathcal O$, if $\mathfrak X(O)=P$, then $P$ is fixed by the group of symmetries of $O$. This is a consequence of the fact that $\mathfrak X$ is equivariant. For every $T\in S(n)$ such that $T(O)=O$, we have:
$$T(P)=T(\mathfrak X(O))=\mathfrak X(T(O))=\mathfrak X(O)=P. $$

Suppose now that $P$ is a point fixed by the group of symmetries of $O$. Let $\mathcal O'$ be the orbit of $O$ in $\mathcal O$, that is, the set of elements $\{T(O):T\in S(n)\}$. Let us begin defining $\mathfrak X$ for the elements in $\mathcal O'$. Let us impose that $\mathfrak X(O)=P$.  For any $O'\in \mathcal O'$ ($O'$ may equal $O$), there exists at least one $T\in G$ such that  $T(O)=O'$, so let us define
$$\mathfrak X(O')=T(\mathfrak X( O))=T(P).$$

\noindent We need to check that this function is well defined. Suppose that there is another element $T'\in S(n)$ such that $T'(O)=O'$. Then, we need that
$$\mathfrak X( O')=T(\mathfrak X( O))=T'( \mathfrak X(O)).$$

\noindent Note that
\begin{align*}
T(O)=T'(O)\Leftrightarrow O=T^{-1}(T'(O)), \\
T(\mathfrak X(O))=T'(\mathfrak X(O))\Leftrightarrow P=T^{-1}(T'(P)).
\end{align*}

\noindent According to the first statement in this proof, {$O=T^{-1}(T'(O))$ implies $ P=T^{-1}(T'(P))$ and so $\mathfrak X\mid_{\mathcal O'}$ is well defined.}

If the action of $S(n)$ over $\mathcal A$ were {transitive} (that is, $\mathcal O=\mathcal O'$) then we would have finished. We would have also finished if there were a \emph{simple center} $\mathfrak C':\mathcal O\to\mathbb R^n$ (such as the centroid) defined in $\mathcal O$ (in this case, for every $O\in\mathcal O\setminus\mathcal O'$, we just impose $\mathfrak X(O)=\mathfrak C(O)$). In other case, to extend $\mathfrak X$ to $\mathcal O$, we need to use the \emph{Axiom of Choice} as follows:

\begin{itemize}

\item Let $\{G_i\}_{i\in I}$ denote the set of subgroups of $S(n)$ with a non-empty set of fixed points in $\mathcal O$. Consider also the set $\{\mathcal X_i\}_{i\in I}$, such that $\mathcal X_i$ is the set of points in $\mathbb R^n$ fixed by $G_i$. Let $\Phi_1$ be a {choice function} that selects an element $P_i$ from each $\mathcal X_i$, for $i\in I$ ($\mathcal X_i$ is never empty).

\item $\mathcal O\setminus \mathcal O'$ can be split into a disjoint union of orbits $\mathcal O\setminus \mathcal O'=\bigcup_{j\in J}\mathcal O_j$. Let $\Phi_2$ be a choice function that selects one element $O_j$ in each orbit $\mathcal O_j$. 

\item If $O$ belongs to the orbit $\mathcal O_j$, then $O=T(O_j)$ for some $T\in S(n)$. If $G_i$ is the stabilizer of $O_j$, for some $i\in I$, we can define
$$\mathfrak X(O)=\mathfrak X(T(O_j))=T(\mathfrak X(O_j))=T(P_i). $$

\noindent The same argument used before shows that this function is well defined and it is equivariant.

\end{itemize}

 \end{proof}

Obviuosly, for the cases $\mathcal O=\mathcal M_m,\mathcal P_m$, the center constructed in the previous proof may fail to be continuous. The problem is still open for continuous centers. This is related (it is more general) to the \emph{Center Conjecture for Equifacetal Simplices} proposed by A. Edmonds (see \cite{E}).

\medskip

\noindent \textbf{Open Question 2.} For $\mathcal O=\mathcal M_m,\mathcal P_m,\mathcal B_d$, decide if, for every space of geometric objects $\mathcal O$ and for every $O\in\mathcal O$, there is a continuous center $\mathfrak X:\mathcal O\to\mathbb R^n$ such that $\mathfrak X(O)=P$ if and only if $P$ is fixed by the group of symmetries of $O$ (that is, for every $T\in S(n)$ such that $T(O)=O$, we have that $T(P)=O$).

\medskip

Another family of statements of Theorem \ref{theorem.es} are the following:

\medskip

\noindent \textbf{Open Question 3.} For two given centers $\mathfrak X_1,\mathfrak X_2:\mathcal O\to\mathbb R^n$, what can we say of $O$ if we know that $\mathfrak X_1(O)=\mathfrak X_2(O)$?

\medskip

\noindent \textbf{Open Question 4.} For particular choices of family of objects $\mathcal O'\subset\mathcal O$, which are closed under the action of $S(n)$, does it exist a family of centers $\mathfrak X_1,\ldots,\mathfrak X_k:\mathcal O\to\mathbb R^n$ such that, for a given  $O\in\mathcal O$, we have that $\mathcal O\in\mathcal O'$ if and only if $\mathfrak X_1(O)=\ldots=\mathfrak X_k(O)$?

\medskip

The answer to these questions depends heavily on the choice of the object space. It is interesting for all of them, but it is only known for very special cases:

\begin{itemize}

\item {For $\mathcal O$ being the set of triangles $\mathcal P_3'$, Statement (a) in Theorem \ref{theorem.es} provides some answers in a particular case. For $\mathfrak X_1$ being the incenter, $\mathfrak X_2$ being the orthocenter and $\mathcal O'$ being the set of equilateral triangles, $\mathfrak X_1(O)=\mathfrak X_2(O)$ if and only if $O\in\mathcal O'$. There exist other results, similar to Theorem \ref{theorem.es}, in the literature.
}

\item {With respect to Open Question 3 and for the case $\mathcal O=\mathcal P_4'$, in Theorems 3.1, 3.3 and 3.4 of \cite{AHK} we can find some interesting examples.
}

\item {In relation to Open Question 4 and for $\mathcal O=\mathcal P_m'$, for $m\geq 3$, we may find some examples in Theorems 16 and 17 of \cite{PS}.}

\end{itemize}

In relation to the previous discussion and for the spaces of geometric objects $\mathcal M_m,\mathcal P_m'$, we have the following:

\medskip

\noindent \textbf{Open Question 5.} For two multiset (resp. polygon) centers $\mathfrak X_1,\mathfrak X_2$ described by an expression analogue to those  in Theorem \ref{theorem.expression}, how can we compute the set of multisets (resp. polygons) satisfying $\mathfrak X_1(O)=\mathfrak X_2(O)?$

\medskip

Note that this problem is very simple for $\mathcal P_3'$: in this case, the coincidence of the centers implies coincidence of the barycentric coordinates appearing in the expression. See, as an example, the following proof of an analogue of Statement (a) in Theorem \ref{theorem.es}:

\begin{theorem} {Let $\mathfrak X$ and $\mathfrak C$ denote, respectively, the Nagel point and the centroid, both defined in $\mathcal P_3'$. For a triangle $\{A,B,C\}\in\mathcal P_3'$, the Nagel point $\mathfrak X(\{A,B,C\})$ coincides with the centroid $\mathfrak C(\{A,B,C\})$ if and only if the triangle $\{A,B,C\}$ is equilateral.}

\end{theorem}

\begin{proof} According to \cite{K}, we have that 
\begin{align*}
\mathfrak X(A,B,C) & =\frac{b+c-a}{a+b+c}A+\frac{a+c-b}{a+b+c}B+\frac{a+b-c}{a+b+c}C, \\
\mathfrak C(\{A,B,C\}) & =\frac{1}{3}A+\frac{1}{3}B+\frac{1}{3}C,
\end{align*}

\noindent where $a,b,c$ denote the length of the sides of $\{A,B,C\}$.

Since
\begin{equation}\label{eq.magdalena}\mathfrak X(\{A,B,C\})=\mathfrak C(\{A,B,C\})\Leftrightarrow \begin{cases}\frac{b+c-a}{a+b+c}=\frac{1}{3}\\\frac{a+c-b}{a+b+c}=\frac{1}{3}\\\frac{a+b-c}{a+b+c}=\frac{1}{3}\end{cases}\end{equation}

\noindent we can see that the only possibility for $\mathfrak X(\{A,B,C\})=\mathfrak C(\{A,B,C\})$ is $a=b=c$.

\end{proof}

The approach followed by Kimberling allows us to produce results like the previous one, with an algebraic proof and with no geometric intuition for $\mathcal P_3'$. But it is not so easy to perform this kind of arguments for other spaces of objects. For example, for $\mathcal P_4'$, for centers with an expression as the one in Equation \eqref{eq.lambda}, the analogue for the statement in Equation \eqref{eq.magdalena} does not hold (because the coefficients $\lambda_1,\lambda_2,\lambda_3,\lambda_4$ appearing in Equation \eqref{eq.lambda} are not coordinates). In this case, Equation \eqref{eq.magdalena} leads to some curious equations of the type
$$\mu_1A+\mu_2B+\mu_3C+\mu_4D=0,\qquad \mu_1+\mu_2+\mu_3+\mu_4=0, $$

\noindent (where $\mu_1,\mu_2,\mu_3,\mu_4$ are functions of $(A,B,C,D)$ satisfying some symmetries) that also define a family of elements in $\mathcal P_4'$  (in a similar fashion to those studied in \cite{BS}). We hope to study this in some further work.

\section{Central Objects} \label{co}

{There are many other constructions that share some properties with centers: they are functions defined over a space of objects $\mathcal O$, that commute with the action of $S(n)$, but whose image may be different from $\mathbb R^n$. We will call them, informally, \emph{central objets}. In the following, we may find some examples. These constructions are as interesting as centers, but are beyond the scope of this work.
}

\begin{exam} Let us start with \textbf{central lines.} These objects are functions $\mathfrak L$ whose codomain is the set of lines in $\mathbb R^2$. Kimberling already studied these objects for the space of triangles in \cite{K2} and some information for the space of $n$-gons can be found in \cite{FP}.

Note that, for any set of objects $\mathcal O$ and for any two centers $\mathfrak X_1,\mathfrak X_2:\mathcal O\to\mathbb R^n$ such that for all $O\in\mathcal O$, $\mathfrak X_1(O)\neq \mathfrak X_2(O)$, we can define a central line by imposing  $\mathfrak L(O) $ to be the line passing through $\mathfrak X_1(O), \mathfrak X_2(O)$.

The converse, in general, is not true. In \cite{FP} the following example is discussed in detail: if $\mathcal O$ is the set of rectangles that are not squares, we can define a central line $\mathfrak L$, but, on the other hand, all the centers defined in $\mathcal O$ must coincide (according to Theorem \ref
{theo.equifacetal} in the present paper), so $\mathfrak L$ cannot be defined as the line passing through two centers.

\end{exam}

\begin{exam} We also have \textbf{central automorphisms.} In this case, these objects are functions $\mathfrak P:\mathcal O\to\mathcal O$. For the case $\mathcal O=\mathcal P_m$, some aspects have been studied in \cite{E}. The most illustrative example is $\mathfrak P:\mathcal P_m'\to\mathcal P_m'$ given by $[(V_1,\ldots, V_n)]\longmapsto [(V_1',\ldots, V_n')]$, where, for $i=1,\ldots, n$, $V_i'$ is the midpoint of $V_i,V_{i+1}$ ($V_{n+1}$ is considered to be $V_1$).

Of course, central automorphisms lead us to a lot of interesting problems, such as the existence of fixed points, sequences of objects defined {by} iterating a given central automorphism and many others.
\end{exam}

%







\bigskip
\bigskip
\begin{minipage}[t]{10cm}
\begin{flushleft}
\small{
\textsc{M. Magdalena Martínez-Rico}
\\*e-mail: magdalenamartinez@educa.madrid.org
\\[0.4cm]
\textsc{L. Felipe Prieto-Martínez}
\\*Universidad Politécnica de Madrid,
\\*Escuela Superior de Arquitectura de Madrid
\\*Av. Juan de Herrera 4
\\* Madrid, 28040, Spain
\\*e-mail: luisfelipe.prieto@upm.es
\\[0.4cm]
\textsc{Raquel Sánchez-Cauce}
\\*e-mail: r.cauce@gmail.com
}
\end{flushleft}
\end{minipage}

\end{document}